\documentclass[a4paper,11pt,reqno]{amsart}
\usepackage{amssymb,amsmath,hyperref,comment}

\usepackage{graphicx}

\title[A generalization of Bressoud's beautiful bijection]{A generalization of Bressoud's beautiful bijection}

\author{Katya Borodinova}
\address{Department of Mathematics and Computer Science, Saint Petersburg State University, Saint Petersburg,  Russia, 199178}
\email{katyaborodinova@gmail.com}

\renewcommand\theta{\vartheta}

\newtheorem{theorem}{Theorem}
\newtheorem{lemma}[theorem]{Lemma}

\theoremstyle{definition}
\newtheorem{definition}[theorem]{Definition}

\numberwithin{theorem}{section} 
\numberwithin{equation}{section}

\makeatletter
\@namedef{subjclassname@2020}{\textup{2020} Mathematics Subject Classification}
\makeatother

\begin{document}

\date{12 August 2026}



\begin{abstract}
In the storied history of the Theory of Partitions, it has long been a quest to find a purely combinatoric proof of the Rogers--Ramanujan identities.   In 1981, Garsia and Milne found a bijective proof using their Garsia--Milne involution principle.  However, their solution is algorithmic and long.  In September 1976, Andrews and Askey presented an open problem related to the combinatoric interpretation of the Rogers--Ramanujan identities at the NATO Advanced Study Institute, Higher Combinatorics, in Berlin.  Askey and Andrews' open problem arose by comparing the partition-theoretic interpretations of one of the Rogers--Ramanujan identities with a related $q$-hypergeometric identity of Rogers (1894).  Ideally, finding the conjectured bijection between the two families of partitions would serve as a more tractable stepping-stone in the quest.  Word spread, and the conjectured bijective proof was quickly found by the then PhD student David Bressoud (1976).  Although Andrews later published Bressoud's beautiful bijection twice in (1986) and (2004),  no one was able to find a generalization of Bressoud's result or even find a low-level analog.  Here we present multiple generalizations of the original problem of Andrews and Askey as well as Bressoud's beautiful bijection.  

\end{abstract}

\maketitle


\section{Introduction and State of Results}

We recall the notion of a partition.  A partition of a positive integer $n$ is a sequence of integers $\lambda=(\lambda_1, \lambda_2, \dots, \lambda_s)$ such that
\begin{equation*}
\lambda_1+\lambda_2+\dots + \lambda_s=n,
\end{equation*}
and 
\begin{equation*}
\lambda_1\ge \lambda_2 \ge \dots \ge \lambda_s\ge 1.
\end{equation*}
For example, $\lambda=(4,4,2,1,1)$ is a partition of the number 12. 

Let us consider the function $p : \mathbb{N}_0 \to \mathbb{N}$, where $p(n)$ is the number of distinct partitions of $n$,where we assume that $p(0)=1$. For example, all partitions of the number $4$ are: 
\begin{equation*}
(4),\ (3,1),\ (2,2),\ (2,1,1),\ (1,1,1,1).
\end{equation*}
\noindent Therefore, $p(4)=5$.

After defining the function $p$, we can work with its properties: it is not difficult to see that it is strictly monotone increasing, and we can find its asymptotic behaviour by estimating it in various ways.

We will consider expressions of the form $p(n | <\text{condition}>)$, meaning by this the number of partitions of $n$ satisfying a certain condition. For example, 
$$p(4|\text{even number of parts})=3,$$
 $$p(4|\text{all parts are odd})=2,$$
  $$p(4|\text{all parts are distinct})=2.$$

This brings us to arguably the most famous results in $q$-series:  the Rogers--Ramanujan identities.  These two identities provide a link between basic hypergeometric series and integer partitions.  The two identities read
\begin{align}
1+\sum_{n=1}\frac{q^{n^2}}{(1-q)(1-q^2)\cdots(1-q^n)}
&=\prod_{n=0}^{\infty}\frac{1}{(1-q^{5n+1})(1-q^{5n+4})},\label{equation:RR1}\\
1+\sum_{n=1}\frac{q^{n^2+n}}{(1-q)(1-q^2)\cdots(1-q^n)}
&=\prod_{n=0}^{\infty}\frac{1}{(1-q^{5n+2})(1-q^{5n+3})}\label{equation:RR2}.
\end{align}
The partition-theoretic interpretation for the first Rogers--Ramanujan identity (\ref{equation:RR1}) reads

\smallskip
{\em The number of partitions of $n$ such that the adjacent parts differ by at least $2$ is equal to the number of partitions of $n$ such that each part is congruent to either $1$ or $4$ modulo $5$,}

\smallskip
\noindent and the partition-theoretic interpretation for the second identity (\ref{equation:RR2}) reads

\smallskip
{\em The number of partitions of $n$ such that the adjacent parts differ by at least $2$ and such that the smallest part is at least $2$ is equal to the number of partitions of $n$ such that each part is congruent to either $2$ or $3$ modulo $5$.}

\smallskip
In the storied history of the Theory of Partitions, it has long been a quest to find a purely combinatoric proof of the Rogers--Ramanujan identities; in other words, on wants to give a combinatoric proof to the partition-theoretic interpretations of (\ref{equation:RR1}) and (\ref{equation:RR2}).   In 1981, Garsia and Milne found a bijective proof using the Garsia--Milne involution principle.  However, their solution is algorithmic and long.  One still wonders if a more insightful bijection is possible. 

\smallskip
In 1976, Andrews and Askey presented an open problem related to the combinatoric interpretation of the Rogers--Ramanujan identities at the NATO Advanced Study Institute, Higher Combinatorics, in Berlin.   Andrews and Askey's open problem arose by comparing the partition-theoretic interpretations of one of the Rogers--Ramanujan identities with a related $q$-hypergeometric identity of Rogers (1894).  Ideally, finding the conjectured bijection between the two families of partitions would serve as a more tractable stepping-stone in the quest.  

\smallskip
To state Andrews and Askey's open problem we recall two identities of Rogers (1894) and their respective partition-theoretic interpretations.   The two identities of Rogers read:
\begin{align}
\left  \{  \prod_{m=1}^{\infty} (1+q^{2m})\right \}&
\left \{1+\sum_{n=1}^{\infty}\frac{q^{n^2}}{(1-q^4) (1-q^{8})\cdots (1-q^{4n})} \right \}\notag \\
&=\prod_{n=0}^{\infty}\frac{1}{(1-q^{5n+1})(1-q^{5n+4})},\label{equation:LJR1}
\end{align}
and
\begin{align}
\left  \{  \prod_{m=1}^{\infty} (1+q^{2m})\right \}&
\left \{1+\sum_{n=1}^{\infty}\frac{q^{n^2+2n}}{(1-q^4) (1-q^{8})\cdots (1-q^{4n})}\right \}\notag\\
&=\prod_{n=0}^{\infty}\frac{1}{(1-q^{5n+2})(1-q^{5n+3})}.\label{equation:LJR2}
\end{align}
The partition-theoretic interpretations of (\ref{equation:LJR1}) and (\ref{equation:LJR2}) then read respectively

\begin{theorem}\cite[Theorem 6]{AndAsk}  The number of partitions of $n$ with distinct parts and with each even part larger than twice the number of odd parts equals the number of partitions of $n$ into parts congruent to $1$ or $4$ modulo $5$.
\end{theorem}

\begin{theorem}\cite[Theorem 7]{AndAsk} The number of partitions of $n$ into distinct parts each larger than $1$ in which each even part is larger than twice the number of odd parts equals the number of partitions of $n$ into parts congruent to $2$ or $3$ modulo $5$.
\end{theorem}

Comparing the partition-theoretic interpretations of (\ref{equation:RR1}) and (\ref{equation:LJR1}) is what motivates the problem of Andrews and Askey \cite{AndAsk}:

\medskip
\noindent \underline{\bf Problem 1.}  {\em  Provide a bijection between the partitions of $n$ in which the parts differ by at least $2$ and the partitions of $n$ into distinct parts in which each even part is twice the number of odd parts.}  

\medskip
Although the following is not explicitly stated in the work of Andrews and Askey \cite{AndAsk}, one could just as easily compare the partition-theoretic interpretations of (\ref{equation:RR2}) and (\ref{equation:LJR2}) to motivate

\medskip
\noindent \underline{\bf Problem 2.} 
{\em Provide a combinatoric bijection between the partitions of $n$ such that the adjacent parts differ by at least $2$ and such that the smallest part is at least $2$ and the partitions of $n$ into distinct parts each larger than $1$ in which each even part is larger than twice the number of odd parts.}

\medskip
After Berlin, Andrews presented his results at Temple University, Philadelphia, and the conjectured bijective proof of Problem $1$ was quickly found by the then PhD student David Bressoud (1976).  Andrews later published Bressoud's beautiful bijection twice (1986),  (2004).  However, no one was able to find a generalization of Bressoud's result or even find a low-level analog.  

\smallskip
To state our results, we will need to introduce some more terminology:
\begin{definition}
    A partition will be called \textit{distinct} or \textit{1-distinct} if all its parts are distinct, i.e., they differ by at least $1$. A partition will be called \textit{super-distinct} or \textit{2-distinct} if any two parts differ by at least $2$.
\end{definition}

\begin{definition}
    A partition will be called \textit{$d$-distinct} if any two parts differ by at least $d$.
\end{definition}

\smallskip
Our approach is motivated by two exercises found in Andrews and Eriksson \cite[Exercises 37, 38]{AE}.

\smallskip
In this paper, we prove four bijective results.  For the first result Theorem \ref{theorem:BressoudAlt}, we establish an analogue of Problem 1 where the roles of odd parts and even parts interchanged.  This is not new and is asked for in Exercise 37 \cite{AE}.  Exercise 38 \cite{AE} asks as a mini-research project to find an analog of Problem 1 but for $3$-distinct partitions.  For our first new result Theorem \ref{theorem:Problem1Gen}, we give a bijective proof for an analog of Problem 1 but for $d$-distinct partitions.  For our second new result Theorem \ref{theorem:Problem2Gen}, we then give a bijective proof for an analog of Problem 2 but for $d$-distinct partitions.  For our third new result Theorem \ref{theorem:BridgeResult}, we provide a theorem that links the generalizations of Problems 1 and 2.

Our first result reads

\begin{theorem} \label{theorem:BressoudAlt} For any natural number $n$, the following holds:
\begin{align*}
&p(n|\text{2-distinct partitions})\\
&\qquad=p(n |\text{distinct partitions and }\\
&\qquad \qquad \qquad  \text{\textbf{any odd part is greater than twice the number of even parts}}).
\end{align*}
\end{theorem}

Thus, the new theorem describes the number of 2-distinct partitions of $n$ in a different way: in terms of the number of 1-distinct partitions, and the condition \textbf{``every even part is greater than twice the number of odd parts"} has been replaced by \textbf{``every odd part is greater than twice the number of even parts"}.

Our first new result is a generalization of Problem 1, and it reads
\begin{theorem}\label{theorem:Problem1Gen} Let $n,\ d$ be natural numbers. Consider the residues modulo $d$ ($0,\ 1,\ \dots\ ,\ d-1$) and any permutation of them: $\pi(0),\ \pi(1),\ \dots\ ,\ \pi(d-1)$, i.e., these same numbers in some other order, each appearing once. Then
\begin{align*}
p(n|\text{$d$-distinct})=p(n|\textit{1-distinct and conditions $c_1,\dots,c_{d-1}$ hold}),
\end{align*}
where the conditions $c_{i}$ are
$$c_1:\ \text{any part $\equiv \pi(1)$ (mod $d$) is greater than}$$
$$\textit{$d * \left ( \text{the total number of parts }\equiv \pi (0) \pmod{d} \right ),$}$$

$$c_2:\ \text{any part $\equiv \pi(2)$ (mod $d$) is greater than}$$
$$\textit{  $d * \left ( \text{the total number of parts }\equiv \pi (0) \text{ or } \pi(1) \pmod{d} \right ),$}$$

$$c_3:\ \text{any part $\equiv \pi(3)$ (mod $d$) is greater than }$$
$$\text{$d * \left ( \text{the total number of parts }\equiv \pi (0),\ \pi (1)\text{ or } \pi(2) \pmod{d} \right ),$}$$

$$\vdots$$

$$c_{d-1}:\ \text{any part $\equiv \pi(d-1)$ (mod $d$) is greater than }$$
$$\text{$d * \left ( \text{the total number of parts }\equiv \pi (0),\ \pi (1),\  \dots\ \text{ or } \pi(d-2) \pmod{d} \right ).$}$$

\end{theorem}

Thus, the second theorem generalizes Bressoud's identity further: for any chosen ordering $\pi(0),\pi(1),\dots,\pi(d-1)$ of the residue classes modulo $d$, the number of $d$-distinct partitions of $n$ equals the number of 1-distinct partitions of $n$ such that for each $s=1,\dots,d-1$, every part congruent to $\pi(s)$ modulo $d$ is greater than $d$ times the total number of parts whose residues belong to $\{\pi(0),\pi(1),\dots,\pi(s-1)\}$. In other words, parts of the ``later" residues must be sufficiently large relative to the count of parts of the ``earlier" residues.

\smallskip
Our second new result is a generalization of Problem 2.  It reads
\begin{theorem}\label{theorem:Problem2Gen} Let $n,\ d$ be natural numbers. Consider the residues modulo $d$ ($0,\ 1,\ \dots\ ,\ d-1$) and any permutation of them: $\pi(0),\ \pi(1),\ \dots\ ,\ \pi(d-1)$. Then
\begin{align*}
p(n|\text{$d$-distinct, smallest part $\geq d$})=p(n|\textit{1-distinct, smallest part $\geq d$ and conditions $c^d_1,\dots,c^d_{d-1}$ hold}),
\end{align*}
where the conditions $c^d_{i}$ are
$$c^d_1:\ \text{any part $\equiv \pi(1)$ (mod $d$) is greater than or equal to}$$
$$\textit{$d * \left (1+ \text{the total number of parts }\equiv \pi (0) \pmod{d} \right ),$}$$

$$c^d_2:\ \text{any part $\equiv \pi(2)$ (mod $d$) is greater than or equal to}$$
$$\textit{  $d * \left ( 1+\text{the total number of parts }\equiv \pi (0) \text{ or } \pi(1) \pmod{d} \right ),$}$$

$$c^d_3:\ \text{any part $\equiv \pi(3)$ (mod $d$) is greater than or equal to}$$
$$\text{$d * \left (1+ \text{the total number of parts }\equiv \pi (0),\ \pi (1)\text{ or } \pi(2) \pmod{d} \right ),$}$$

$$\vdots$$

$$c^d_{d-1}:\ \text{any part $\equiv \pi(d-1)$ (mod $d$) is greater than or equal to}$$
$$\text{$d * \left (1+ \text{the total number of parts }\equiv \pi (0),\ \pi (1),\  \dots\ \text{ or } \pi(d-2) \pmod{d} \right ).$}$$

\end{theorem}

Our third new result forms a bridge between Theorems \ref{theorem:Problem1Gen} and \ref{theorem:Problem2Gen}.  It reads
\begin{theorem}\label{theorem:BridgeResult} 

Let $n,\ d,\ k\geq1$ be natural numbers. Consider the residues modulo $d$ ($0,\ 1,\ \dots\ ,\ d-1$) and any permutation of them: $\pi(0),\ \pi(1),\ \dots\ ,\ \pi(d-1)$. Then
\begin{align*}
p(n|\text{$d$-distinct, smallest part $\geq k$})=p(n|\textit{1-distinct, smallest part $\geq k$ and conditions $c^k_1,\dots,c^k_{d-1}$ hold}),
\end{align*}
where the conditions $c^k_{i}$ are
$$c^k_1:\ \text{any part $\equiv \pi(1)$ (mod $d$) is greater than or equal to}$$
$$\textit{$k+d * \left (\text{the total number of parts }\equiv \pi (0) \pmod{d} \right ),$}$$

$$c^k_2:\ \text{any part $\equiv \pi(2)$ (mod $d$) is greater than or equal to}$$
$$\textit{  $k+d * \left ( \text{the total number of parts }\equiv \pi (0) \text{ or } \pi(1) \pmod{d} \right ),$}$$

$$c^k_3:\ \text{any part $\equiv \pi(3)$ (mod $d$) is greater than  or equal to}$$
$$\text{$k+d * \left (\text{the total number of parts }\equiv \pi (0),\ \pi (1)\text{ or } \pi(2) \pmod{d} \right ),$}$$

$$\vdots$$

$$c^k_{d-1}:\ \text{any part $\equiv \pi(d-1)$ (mod $d$) is greater than or equal to}$$
$$\text{$k+d * \left (\text{the total number of parts }\equiv \pi (0),\ \pi (1),\  \dots\ \text{ or } \pi(d-2) \pmod{d} \right ).$}$$

\end{theorem}

In Section 2 we introduce Young Diagrams. In Section 3 we prove Theorem 1.5. In Section 4, we prove Theorem 1.6. In Section 5, we prove Theorem 1.8.

\section{An introduction to Young diagrams}\label{section:YoungIntro}

    A very useful tool when working with partitions are \textit{Young diagrams}, which are diagrams consisting of several rows, arranged in non-increasing order from top to bottom. The number of boxes in a row represents one part of the partition. For example, this is how the partition $\lambda=(5,4,1)$ of the number $10$ can be represented as a Young diagram:

    \begin{center}
        \includegraphics[width=4cm]{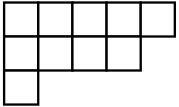}
    \end{center}

Young tableaux make some previously non-obvious statements evident, for example:

\begin{lemma} For any natural numbers $m$, $n$:
    $$p(n \mid \text{the number of parts does not exceed } m) = p(n \mid \text{each part does not exceed } m).$$
\end{lemma}
\begin{proof}
    Consider a partition $\lambda$ of the number $n$ with at most $m$ parts. Construct the Young diagram corresponding to this partition. Its height is equal to the number of parts, it does not exceed $m$. Reflect this diagram about the main diagonal. It is easy to see that the result is again a Young diagram, and its top row (i.e., the longest one) does not exceed $m$. Therefore, the partition corresponding to this diagram consists of parts not exceeding $m$. Thus, we have constructed a bijection between partitions of $n$ with at most $m$ parts and partitions of the same number where all parts are at most $m$. This proves the required equality between the cardinalities of these sets.
\end{proof}

\section{Proof of Theorem \ref{theorem:BressoudAlt}}\label{section:BressoudAlt}

\textbf{Idea of proof}
    The proof consists of constructing a bijection that deforms the Young diagram of a given partition into 2-distinct parts into the Young diagram of a partition into distinct parts, where each even part is greater than twice the number of odd parts. One of the steps in constructing the bijection is placing the odd parts at the top and the even parts at the bottom.\\

We will construct a bijection between the sets of partitions of the number $n$ with the corresponding properties, thereby proving the equality.

Consider a partition of the number $n$ into super-distinct parts:
\[
a_1 > a_2 > \dots > a_m, \quad a_i - a_{i+1} \ge 2.
\]

\noindent \textbf{Step 1} Construct a Young diagram. In it, any two adjacent rows differ by at least two, and the bottom (i.e., the smallest) row is $1$ or greater. Then evaluate each of the rows in the diagram:
$$a_m\geq1$$
$$a_{m-1}\geq a_m+2\geq3$$
$$a_{m-2}\geq a_{m-1}+2\geq5$$
$$\vdots$$
$$a_2\geq a_3+2\geq 2m-3$$
$$a_1\geq a_2+2\geq 2m-1$$

Then subtract $1$ from $a_m$, $3$ from $a_{m-1}$, $5$ from $a_{m-2}$, $\dots$, $2m-3$ from $a_2$, $2m-1$ from $a_1$:

$$b_m=a_m-1$$
$$b_{m-1}=a_{m-1}-3$$
$$b_{m-2}=a_{m-2}-5$$
$$\vdots$$
$$b_2=a_2-(2m-3)$$
$$b_1=a_1-(2m-1)$$

First, note that the resulting numbers are non-negative due to the inequalities above. Second, $b_1 \geq b_2 \geq b_3 \geq \ldots \geq b_{m-1} \geq b_m\geq0$:

$$a_i\geq a_{i+1}+2$$
$$b_i=a_i-2(m-i)-1$$
$$b_{i+1}=a_{i+1}-2(m-(i+1))-1$$
$$b_i-b_{i+1}=(a_i-2(m-i)-1)-(a_{i+1}-2(m-(i+1))-1)=a_i-a_{i+1}-2\geq0$$

Thus, we can shift each row to the left by the corresponding number of boxes and obtain, to the left of the separator, a Young diagram with parts $1,\ 3,\ 5,\ \dots\ ,\ 2m-3,\ 2m-1$, and to the right, a Young diagram for the partition $(b_1,\dots,b_m)$ (more precisely, the last few parts may become zero after subtraction, but we will assume there are exactly $m$ parts, some of which may be zero):

\begin{center}
    \includegraphics[width=10cm]{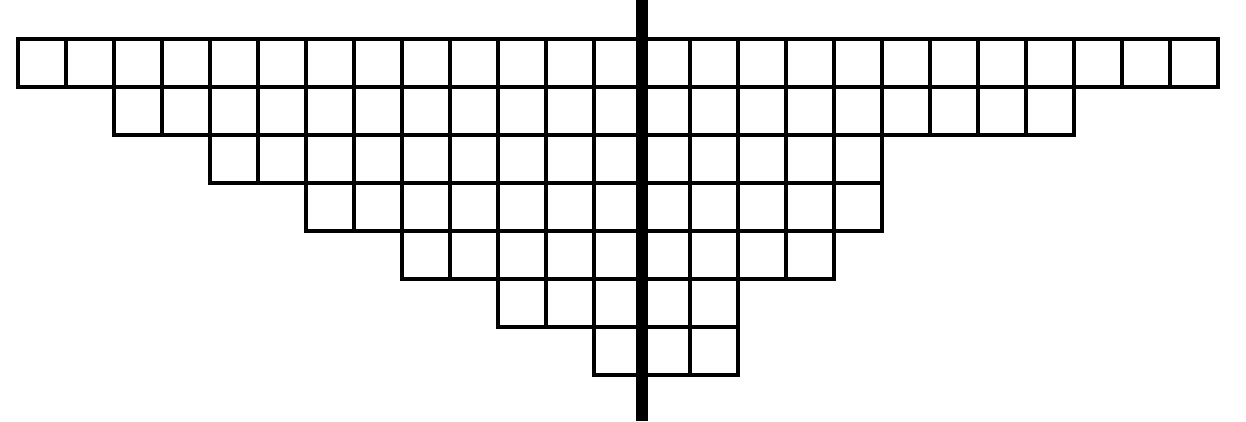}
\end{center}
\noindent \textbf{Step 2} Sort the parts $b_i$ in non-increasing order from top to bottom, dividing them into two groups by parity: even ones on top, odd ones on the bottom:

\begin{center}
    \includegraphics[width=10cm]{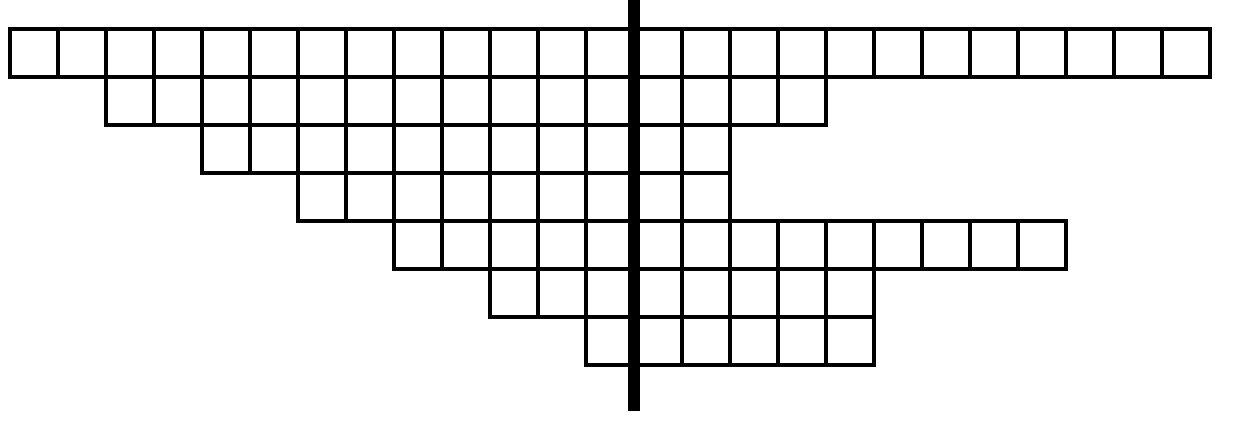}
\end{center}

\noindent \textbf{Step 3} Now shift each row of the diagram back to the right, i.e., add to the odd parts, in increasing order, the numbers $1,\ 3,\ 5,\ \dots\ ,\ 2k-1$, and then to the even parts, in increasing order, $2k+1,\ 2k+3,\ \dots\ , 2m-1$:

\begin{center}
    \includegraphics[width=9cm]{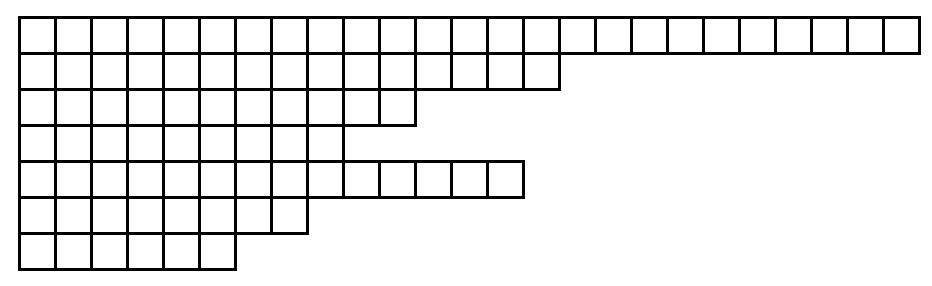}
\end{center}
\noindent \textbf{Step 4} 

Note that after Step 2 on the right side and after Step 3, we do not always obtain Young diagrams, i.e., the parts are not sorted. Therefore, as a final step, we sort the parts:

\begin{center}
    \includegraphics[width=9cm]{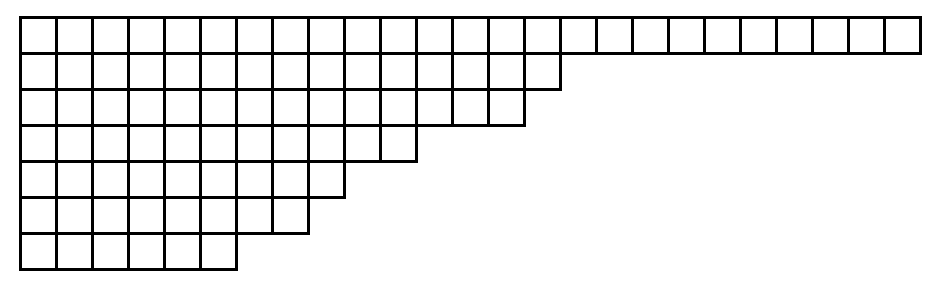}
\end{center}

Thus, in 4 steps we have transformed a super-distinct partition of the number $n$ into some partition of the number $n$. Now we will prove that the resulting partition will satisfy the required properties: it will consist of distinct parts, and any even part (if they exist) is greater than twice the number of odd parts.\\

First, let us understand that the result is always a 1-distinct partition: indeed, in Step 3 we add consecutive odd numbers to monotonically increasing odd parts and monotonically increasing even parts. Therefore, even parts become odd, and odd parts become even, and if two parts of the same parity were unequal, the difference between them increases, while if they were equal, they become unequal after adding distinct odd numbers.\\

It remains to prove the inequality for the odd parts. For this, assume that a distinct partition $\pi$ was obtained from some 2-distinct partition using the algorithm described above. We will construct the inverse algorithm and determine for which partitions it can be executed.\\

First, we reverse Step 4. After performing Step 3 of the initial algorithm, the top consists of a group of odd parts sorted in strictly decreasing order, and the bottom consists of a group of even parts sorted in strictly decreasing order. Therefore, to reverse Step 4, we need to rearrange the parts of the partition $\pi$, placing the odd parts at the top and the even parts at the bottom.

\begin{center}
    \includegraphics[width=10cm]{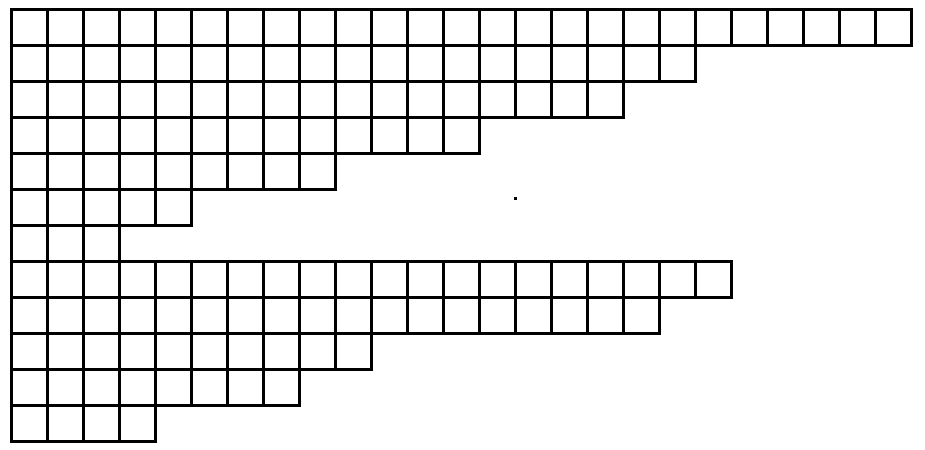}
\end{center}

Let $\pi_o$ denote the number of odd parts and $\pi_e$ denote the number of even parts. Now reverse Step 3: subtract $1$ from the smallest even part, $3$ from the second smallest even part, $\dots$, $2\pi_e - 1$ from the largest even part; then subtract $2\pi_e + 1$ from the smallest odd part, $2\pi_e + 3$ from the second smallest odd part, $\dots$, $2(\pi_e + \pi_o) - 1$ from the largest odd part.

Note that since $\pi$ is a 1-distinct partition, any two parts of the same parity differ by at least $2$, and they are sorted in non-increasing order. Therefore, if it is possible to subtract $1$ from the smallest even part (which can always be done), then the consecutive odd numbers can be subtracted from the remaining, larger even parts. And if it is possible to subtract $2\pi_e + 1$ from the smallest odd part, then the subsequent consecutive odd numbers can be subtracted from the remaining, larger odd parts. Thus, the only potential issue is the inability to subtract $2\pi_e + 1$ from the smallest odd part. This can be done exactly when this smallest odd part, and consequently all others, is at least $2\pi_e + 1$, meaning strictly greater than twice the number of even parts.

\begin{center}
    \includegraphics[width=10cm]{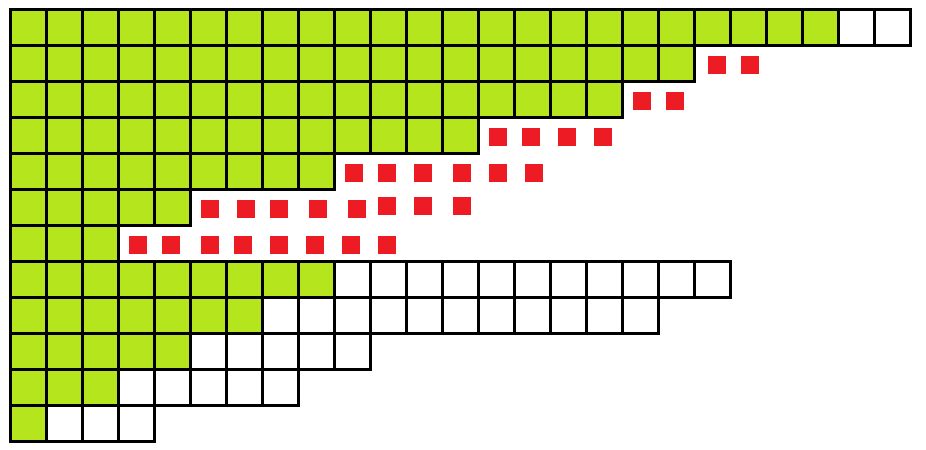}
\end{center}

Next, if the required inequality holds, we shift the parts to the left by the consecutive odd numbers $1,\ 3,\ \dots\ ,\ 2(\pi_e+\pi_o)-1$, leaving unsorted parts on the right, and then reverse Steps 2 and 1 by sorting the parts on the right and shifting them back. It remains to explain why the result after inversion is a 2-distinct partition. This follows from the fact that, when reversing Step 1 after reversing Step 2, we add strictly increasing odd parts to monotonically increasing right-hand parts.\\

Thus, we have constructed a direct algorithm that transforms a 2-distinct partition into some other partition, and then we reversed the algorithm, showing that, firstly, it can be reversed only for partitions with the required properties, and secondly, the reversed algorithm yields a 2-distinct partition. Therefore, this is a bijection between these two sets, which proves the equality:

$$p(n|\text{2-distinct partitions})=$$
$$=p(n|\textit{distinct partitions and any odd part is greater than twice the number of even parts}).$$

\section{Proof of Theorem \ref{theorem:Problem1Gen}}\label{section:Problem1Gen}

We will construct a bijection between the sets of partitions of the number $n$ with the corresponding properties, thereby proving the equality.

Consider a partition of the number $n$ into $d$-distinct parts:
\[
a_1 > a_2 > \dots > a_m, \quad a_i - a_{i+1} \geq d.
\]

\noindent \textbf{Step 1} Construct a Young diagram. In it, any two adjacent rows differ by at least $d$, and the bottom (i.e., the smallest) row is $1$ or greater. Then evaluate each of the rows in the diagram:
$$a_m\geq1$$
$$a_{m-1}\geq a_m+d\geq d+1$$
$$a_{m-2}\geq a_{m-1}+d\geq 2d+1$$
$$\vdots$$
$$a_2\geq a_3+d\geq (m-2)d+1$$
$$a_1\geq a_2+d\geq (m-1)d+1$$

Then subtract $1$ from $a_m$, $d+1$ from $a_{m-1}$, $2d+1$ from $a_{m-2}$, $\dots$, $(m-2)d+1$ from $a_2$, $(m-1)d+1$ from $a_1$:

$$b_m=a_m-1$$
$$b_{m-1}=a_{m-1}-(d+1)$$
$$b_{m-2}=a_{m-2}-(2d+1)$$
$$\vdots$$
$$b_2=a_2-((m-2)d+1)$$
$$b_1=a_1-((m-1)d+1)$$

First, note that the resulting numbers are non-negative due to the inequalities above. Second, $b_1 \geq b_2 \geq b_3 \geq \ldots \geq b_{m-1} \geq b_m\geq0$:

$$a_i\geq a_{i+1}+d$$
$$b_i=a_i-((m-i)d+1)$$
$$b_{i+1}=a_{i+1}-((m-(i+1))d+1)$$
$$b_i-b_{i+1}=(a_i-((m-i)d+1))-(a_{i+1}-((m-(i+1))d+1))=a_i-a_{i+1}-d\geq0$$

Thus, we can shift each row to the left by the corresponding number of boxes and obtain, to the left of the separator, a Young diagram with parts $1,\ d+1,\ 2d+1,\ \dots\ ,\ ((m-2)d+1),\ ((m-1)d+1)$, and to the right, a Young diagram for the partition $(b_1,\dots,b_m)$ (more precisely, the last few parts may become zero after subtraction, but we will assume there are exactly $m$ parts, some of which may be zero):

\begin{center}
    \includegraphics[width=10cm]{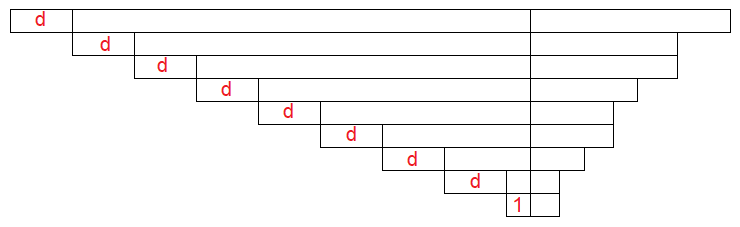}
\end{center}
\noindent \textbf{Step 2}

Sort the parts $b_i$, forming $d$ groups, each consisting of non-increasing parts, sorted from top to bottom. The parts are divided into these $d$ groups according to their residues modulo $d$: at the top, those congruent to $\pi(d-1)-1$ modulo $d$; next, those congruent to $\pi(d-2)-1$ modulo $d$; $\dots$; followed by those congruent to $\pi(1)-1$ modulo $d$; and finally, at the bottom, those congruent to $\pi(0)-1$ modulo $d$:

\begin{center}
    \includegraphics[width=10cm]{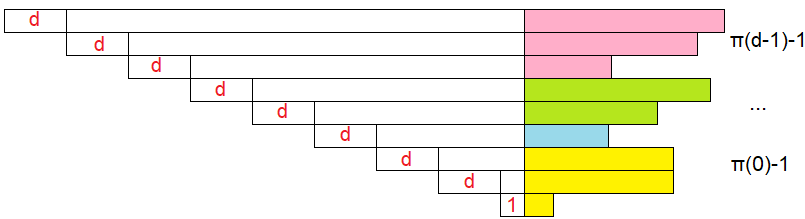}
\end{center}

\noindent \textbf{Step 3} Now shift each row of the diagram back to the right, i.e., add to the parts $\equiv \pi(0)-1$, in increasing order, the numbers $1,\ (d+1),\ (2d+1),\ \dots\ ,\ ((k_0-1)d+1)$; to the parts $\equiv \pi(1)-1$, the numbers $(k_0d+1),\ ((k_0+1)d+1),\ \dots\ ,\ ((k_0+k_1-1)d+1)$; $\dots$; to the parts $\equiv \pi(d-1)-1$, in increasing order, the numbers $((k_0+k_1+\dots+k_{d-2})d+1),\ ((k_0+k_1+\dots+k_{d-2}+1)d+1),\ \dots\ ,\ ((k_0+k_1+\dots+k_{d-2}+k_{d-1}-1)d+1)$, where $k_i$ is the number of parts congruent to $\pi(i)-1$ modulo $d$. Note that after addition, parts congruent to $\pi(i)-1$ modulo $d$ will be congruent to $\pi(i)$ modulo $d$:

\begin{center}
    \includegraphics[width=10cm]{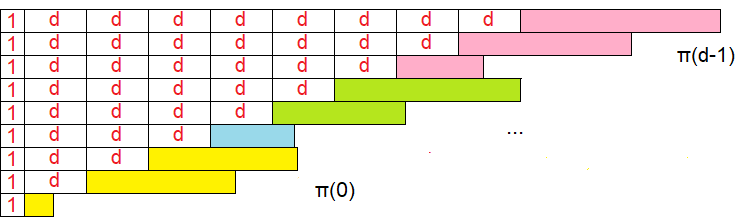}
\end{center}
\noindent \textbf{Step 4} 

Note that after Step 2 on the right side and after Step 3, we do not always obtain Young diagrams, i.e., the parts are not sorted. Therefore, as a final step, we sort the parts:

\begin{center}
    \includegraphics[width=10cm]{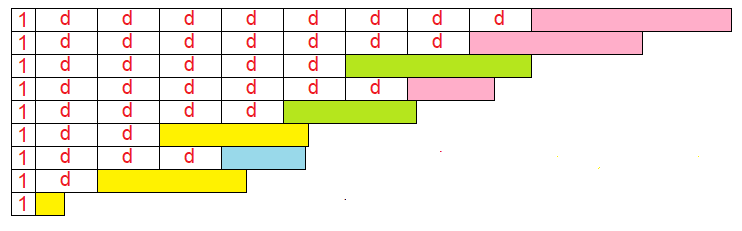}
\end{center}

Thus, in 4 steps we have transformed a $d$-distinct partition of the number $n$ into some partition of the number $n$. Now we will prove that the resulting partition will satisfy the required properties: it will consist of distinct parts and satisfy conditions $c_1,\dots,c_{d-1}$.\\

First, let us understand that the result is always a 1-distinct partition: indeed, in Step 3 we add consecutive numbers congruent to 1 modulo $d$ to monotonically increasing parts, grouped by residues. Therefore, parts that were equal to each other will become unequal, and if they differed, they will differ even more (if they were congruent to each other modulo $d$).\\

It remains to prove properties $c_1,\dots,c_{d-1}$. For this, assume that a distinct partition $\lambda$ was obtained from some $d$-distinct partition using the algorithm described above. We will construct the inverse algorithm and determine for which partitions it can be executed.\\

First, we reverse Step 4. After performing Step 3 of the initial algorithm, at the top we have a group of parts $\equiv \pi(d-1)$ sorted in non-increasing order, followed by parts $\equiv\pi(d-2)$, $\dots$, and at the very bottom, a group of parts $\equiv\pi(0)$ sorted in non-increasing order. Therefore, to reverse Step 4, we need to rearrange the parts of the partition $\pi$ into these groups, keeping them sorted within each group.

\begin{center}
    \includegraphics[width=10cm]{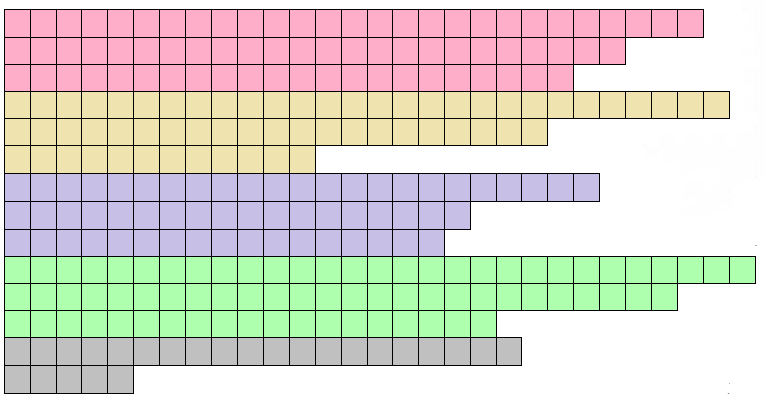}
\end{center}

Let $\lambda_0$ denote the number of parts $\equiv\pi(0)$, $\lambda_1$ the number of parts $\equiv\pi(1)$, $\dots$, and $\lambda_{d-1}$ the number of parts $\equiv\pi(d-1)$. Now reverse Step 3: subtract $1$ from the smallest part in the bottom group of $\lambda_0$ parts, $d+1$ from the second, $\dots$, $d \cdot(\lambda_0-1)+1$ from the largest part in this group; then subtract $d\cdot\lambda_0+1$ from the smallest part in the second group, $d\cdot(\lambda_0+1)+1$ from the second, $\dots$, $d\cdot(\lambda_0+\lambda_1-1)+1$ from the largest part in this group; $\dots$; finally, subtract $d\cdot(\lambda_0+\dots+\lambda_{d-2})+1$ from the smallest part in the top group, $d\cdot(\lambda_0+\dots+\lambda_{d-2}+1)+1$ from the second, $\dots$, $d\cdot(\lambda_0+\dots+\lambda_{d-2}+\lambda_{d-1}-1)+1$ from the largest part in the top group.

Note that since $\lambda$ is a 1-distinct partition, any two parts with the same residue differ by at least $d$, and they are sorted in non-increasing order. Therefore, if it is possible to subtract $1$ from the smallest part congruent to $\pi(0)$ (which can always be done), then the successive numbers of the form $d\cdot l+1$ can be subtracted from the remaining, larger parts of the bottom group. If it is possible to subtract $d\cdot\lambda_0+1$ from the smallest part of the second group, then the subsequent successive numbers can be subtracted from the remaining, larger parts of that group, and so on. Thus, the only possible issues arise if it is impossible to subtract $d\cdot(\lambda_0+\dots+\lambda_{s-1})+1$ from the smallest part of the $(s+1)$-th group from the bottom (i.e., the group where parts are congruent to $\pi(s)$ modulo $d$), for $1 \leq s \leq d-1$. This can be done exactly when this smallest part congruent to $\pi(s)$, and consequently all others, is at least $d\cdot(\lambda_0+\dots+\lambda_{s-1})+1$, i.e., strictly greater than $d\cdot\left(\textit{the total number of parts $\equiv \pi (0),\ \pi(1),\ \dots \textit{ or } \pi(s-1) \pmod{d}$}\right)$.

Next, if the required inequality holds, we shift the parts to the left by the consecutive numbers $1,\ (d+1),\ \dots\ ,\ d\cdot(\lambda_0+\dots+\lambda_{d-1}-1)+1$, leaving unsorted parts on the right, and then reverse Steps 2 and 1 by sorting the parts on the right and shifting them back. It remains to explain why the result after inversion is a $d$-distinct partition. This follows from the fact that, when reversing Step 1 after reversing Step 2, we add strictly increasing numbers, all congruent modulo $d$, to monotonically increasing right-hand parts.\\

Thus, we have constructed a direct algorithm that transforms a $d$-distinct partition into some other partition, and then we reversed the algorithm, showing that, firstly, it can be reversed only for partitions with the required properties, and secondly, the reversed algorithm yields a $d$-distinct partition. Therefore, this is a bijection between these two sets, which proves the equality:

\begin{align*}
p(n|\text{$d$-distinct})
=p(n|\text{1-distinct and conditions $c_1,\dots,c_{d-1}$ hold}).
\end{align*}

\section{Proof of Theorem \ref{theorem:BridgeResult}}\label{section:BridgeResult}

Theorem \ref{theorem:BridgeResult} is a generalization of Theorem \ref{theorem:Problem2Gen}, so we will prove it.\\

The proof of this theorem differs from that of the main theorem in that the first step involves subtracting the staircase $k, k+d, k+2d, \dots$ from the partition --- rather than the staircase $1, d+1, 2d+1, \dots$ --- which is possible due to the $d$-distinctness property and the constraint on the size of the smallest part. Next, the right-hand side of the diagram is again partitioned based on residues; this time into groups congruent to $\pi(0)-k,\ \pi(1)-k, \dots,\ \pi(d-1)-k$ from top to bottom, since the staircase parts are congruent to $k$ modulo $d$ rather than $1$ (and within each group, in non-increasing order). It can be done because $\pi(0)-k$ and so on $\pi(d-1)-k$ is again a complete system of residues. Finally, the staircase is added back, and the resulting parts are sorted, once again satisfying the conditions of $1$-distinctness. Because the smallest step of the staircase is $k$, the smallest part of the resulting partition will be $\geq k$.\\

Now let us see how the conditions $c_1, \dots, c_{d-1}$ have changed. To do this, we attempt to reverse the algorithm for an arbitrary $1$-distinct partition with smallest part $\geq k$. We group the parts according to their residues: $\pi(0)$ at the top, then $\pi(1)$, ..., and $\pi(d-1)$ at the bottom. Now, we need to subtract the staircase $k, k+d, k+2d, \dots$. Given the grouping by residues and the fact that the parts are $1$-distinct, it suffices to subtract the appropriate step of the staircase from the smallest part in each group. Thus, if there are exactly $m$ parts under with the smallest part in group $\pi(i)$, we want to subtract the step $k+md$ from it; this gives rise to the condition
$$c^k_{i}:\ \text{any part $\equiv \pi(i)$ (mod $d$) is greater than or equal to}$$
$$\text{$k+d * \left (\text{the total number of parts }\equiv \pi (0),\ \pi (1),\  \dots\ \text{ or } \pi(i-1) \pmod{d} \right ).$}$$

 After this operation, the remaining parts are sorted, and the staircase is added back, completing the reverse algorithm.

\section{Acknowledgments}

I would like to thank Mr. Mortenson, Mr. Andrews, and Mr. Eriksson for their inspiration in this work.

\end{document}